\documentclass[12pt]{article}
\usepackage[a4paper,margin=1.0in]{geometry}
\usepackage[colorlinks,citecolor=magenta,linkcolor=black]{hyperref}
\pdfpagewidth=\paperwidth \pdfpageheight=\paperheight
\usepackage{amsfonts,amssymb,amsthm,amsmath,eucal,tabu,url}
\usepackage{pgfplots}
\pgfplotsset{compat=1.15}
\usepackage{mathrsfs}
\usetikzlibrary{arrows}
 \usepackage{pgf,tikz}
 \usetikzlibrary{automata}
 \usetikzlibrary{arrows}
 \usepackage{indentfirst}
\usepackage{tabularx} 

\theoremstyle{plain}
\newtheorem{thm}{Theorem}[section]
\newtheorem{theorem}[thm]{Theorem}

\newtheorem{proposition}[thm]{Proposition}

\newtheorem{conjecture}[thm]{Conjecture}

\theoremstyle{definition}
\newtheorem{definition}[thm]{Definition}
\newtheorem{remark}[thm]{Remark}

\newtheorem{thevarthm}[thm]{\varthmname}

\newenvironment{varthm*}[1]{\trivlist\item[]{\bf #1.}\it}{\endtrivlist}

\renewcommand\geq{\geqslant}

\renewcommand\leq{\leqslant}

\newcommand\be{\begin{eqnarray*}}
\newcommand\ee{\end{eqnarray*}}

\newcommand\newop[2]{\def#1{\mathop{\rm #2}\nolimits}}
\newop\log{log}
\newop\ord{ord}
\newop\Gal{Gal}
\newop\SL{SL}
\newop\Bl{Bl}
\newop\mult{mult}
\newop\mass{mass}
\newop\div{div}
\newop\codim{codim}
\newop\sing{sing}
\newop\vdim{vdim}
\newop\edim{edim}
\newop\Ass{Ass}
\newop\size{size}
\newop\reg{reg}
\newop\satdeg{satdeg}
\newop\supp{supp}
\newop\Neg{Neg}
\newop\Nef{Nef}
\newop\Nefh{Nef_H}
\newop\Eff{Eff}
\newop\Zar{Zar}
\newop\MB{MB}
\newop\MBxC{MB\mathit{(x,C)}}
\newop\NnB{NnB}
\newop\Bigg{Big}
\newop\Effbar{\overline{\Eff}}

\def\keywordname{{\bfseries Keywords}}%
\def\keywords#1{\par\addvspace\medskipamount{\rightskip=0pt plus1cm
\def\and{\ifhmode\unskip\nobreak\fi\ $\cdot$
}\noindent\keywordname\enspace\ignorespaces#1\par}}
\def\subclassname{{\bfseries Mathematics Subject Classification
(2020)}\enspace}
\def\subclass#1{\par\addvspace\medskipamount{\rightskip=0pt plus1cm
\def\and{\ifhmode\unskip\nobreak\fi\ $\cdot$
}\noindent\subclassname\ignorespaces#1\par}}

\begin{document}
\title{Freeness of arrangements of lines and one conic with ordinary quasi-homogeneous singularities}
\author{Piotr Pokora}
\date{\today}
\maketitle

\thispagestyle{empty}
\begin{abstract}
The main purpose of the present paper is to provide a partial classification, performed with respect the weak-combinatorics, of free arrangements consisting of lines and one smooth conic with quasi-homogeneous ordinary singularities.
\keywords{conic-line arrangements, quasi-homogeneous singularities, freeness}
\subclass{51B05, 51A45, 14N25, 32S25}
\end{abstract}
\section{Introduction}
The present paper is devoted to arrangements of lines and exactly one conic in the complex projective plane with quasi-homogeneous ordinary singularities. Here by an arrangement we understand a reduced and reducible plane curve while by a reduced curve we mean a curve that might be also irreducible. Our main motivation comes from a very active area of research devoted to free arrangements of rational curve arrangements, where an arrangement is {\it free} if its associated module of derivations is a free module over the coordinate ring of the plane, and the so-called Numerical Terao's Conjecture which focuses on the so-called \textbf{weak combinatorics} of a given arrangement.
\begin{definition}
Let $C = \{C_{1}, ..., C_{k}\} \subset \mathbb{P}^{2}_{\mathbb{C}}$ be a reduced curve such that each irreducible component $C_{i}$ is \textbf{smooth}. The weak combinatorics of $C$ is a vector of the form $(d_{1}, ..., d_{s}; m_{1}, ..., m_{p})$, where $d_{i}$ denotes the number of irreducible components of $C$ of degree $i$, and $m_{j}$ denotes the number of singular points of a curve $C$ of a given type $M_{j}$.
\end{definition}
In the above definition we refer to types of singularities, which for us are determined by their local normal forms. In our paper we will use Arnold's classification of local normal forms presented in \cite{arnold}.

For instance, if $\mathcal{A} \subset \mathbb{P}^{2}_{\mathbb{C}}$ is an arrangement of $d\geq 2$ lines, then the weak combinatorics of $\mathcal{A}$ is $(d, t_{2}, ..., t_{d})$, where $t_{j}$ denotes the number of $j$-fold intersection points of $\mathcal{A}$.
Having this definition in hand, we can formulate the motivating conjecture for the investigations in the present paper.
\begin{conjecture}[Numerical Terao's Conjecture]
Let $C_{1}, C_{2}$ be two reduced curves in $\mathbb{P}^{2}_{\mathbb{C}}$ such that their all irreducible components are smooth. Suppose that $C_{1}$ and $C_{2}$ have the same weak combinatorics and all singularities that our curves admit are quasi-homogeneous. Assume that $C_{1}$ is free, then $C_{2}$ has to be free.
\end{conjecture}
This conjecture is somehow a natural generalization of the classical Terao's conjecture on (central) hyperplane arrangements, where we focus on the intersection posets of arrangements as the decisive objects. It is worth recalling that the author with Alexandru Dimca showed that if $\mathcal{CL}$ is an arrangement of $k\geq 1$ smooth conics and $d\geq 1$ lines that admits nodes, tacnodes, and ordinary triple points as singularities, that the Numerical Terao's Conjecture holds for this class of curves \cite{DimcaPokora}. On the other hand, Marchesi and Vall\'es in \cite{Mar} gave a counterexample to the Numerical Terao's Conjecture in the class of line arrangements or, more precisely, in the class of \textbf{triangular line arrangements}. Very recently, the author with Luber and K\"uhne in \cite{LKP} have shown that Numerical Terao's Conjecture holds for arrangements with up to $d = 12$ lines, and there are some counterexamples to this conjecture for $d=13$ lines. Based on that result, it is somehow natural to ask whether we can find a counterexample to Numerical Terao's Conjecture in the least general setting comparing with the line arrangement scenario. More precisely, we seek for the simplest possible counterexample, if existing, in the setting of smooth rational curve arrangements with ordinary quasi-homogeneous singularities, so we want to work with conic-line arrangements admitting ordinary points of multiplicity $2,3$ and $4$. For the completeness of our presentation we recall that an ordinary singularity of multiplicity $m$ is quasi-homogeneous if $m<5$, and this can be found in \cite[Exercise 7.31]{RCS}. Note also that our assumption of working with quasi-homogeneous singularities follows from the fact that, in general, singularities of conic-line arrangements are not quasi-homogeneous, and this leads to some pathologies as indicated in \cite[Example 4.6]{Poksur}. Another motivation for our studies comes from a completely different perspective that we want to explain now. We have several interesting invariants that can be attached to the Jacobian ideal associated with a defining equation $f \in S:=\mathbb{C}[x,y,z]$ of a reduced plane curve $C \, : f=0$. One such invariant is \textbf{the minimal degree of Jacobian relations}, which is defined as the minimal degree of a non-trivial triple $(a,b,c) \in S^{3}$ satisfying the condition that
$$a\partial_{x} \, f + b \partial_{y} \, f + c \partial_{z} \, f = 0.$$ In this context, it is worth recalling a pair of non-free arrangements of $d=9$ lines constructed by Ziegler in \cite{Ziegler} with the property that the arrangements have the same intersection lattices, so they have the same weak combinatorics which is $(d,t_{2},t_{3}) = (9;18,6)$, but they have different minimal degrees of the Jacobian relations. These two line arrangements are distinguished by the property that in one case the $6$ triple points are on a smooth conic, and in the other case they are not. From our point of view, it is natural to include this ghostly existing conic passing through $6$ points and then to study the homological properties of the resulting conic-line arrangement. 
In the present paper, motivated by Ziegler's example and a recent paper by Dimca and Sticlaru \cite{DiSti}, we want to study the freeness of arrangements consisting of one smooth conic and $d\geq 3$ lines admitting ordinary quasi-homogeneous singularites, in the hope of better understanding the freeness property from the perspective of weak combinatorics and the minimal degree of Jacobian relations. This setting allows to provide a detailed partial classification result on admissible weak combinatorics of free arrangements with $3 \leq d\leq 10$ lines and one conic having $n_{2}$ nodes, $n_{3}$ ordinary triple, and $n_{4}$ ordinary quadruple points as singularities (i.e. singularities of types $A_{1}, D_{4}$, and $X_9$ according to Arnold's classification \cite{arnold}). Since we are working with only one conic, we abbreviate the representation of the vectors of weak combinatorics to the form $(d;n_{2},n_{3},n_{4})$.
\begin{theorem}[Partial Classification]
\label{class}
Let $\mathcal{CL}$ be an arrangement of $3\leq d \leq 10$ lines and one smooth conic in the complex projective plane such that it admits ordinary singularities of multiplicity $<5$. Then the following weak combinatorics can be geometrically realized over the real numbers as \textbf{free} arrangements:
\begin{multline*}
(d;n_{2}, n_{3}, n_{4}) \in \{(3;0,3,0),(3;3,0,1),(4;2,2,1),(5;2,2,2),(5;5,1,2),\\
(6;3,0,4),(6;3,4,2),(6;6,1,3), (7;5,2,4), (7;5,4,3), (7;8,1,4), (8;2,8,3), \\ (8;5,5,4),(8,1;8,2,5), (9;6,4,6), (10;8,1,9)\}.
\end{multline*}
\end{theorem}
Our assumption that $d\geq 3$ follows from the fact that for $d<3$ there are no free arrangements. It is worth emphasizing that our classification is explicit since we provide the defining equations.

During the preparation of the paper I was informed by Tomasz Pe\l ka that there is an interesting thesis which might be worth looking at, so this is also a good moment to notice that our arrangements turn out to have a special meaning in a completely different area. This work is the doctoral thesis of M. Neusel \cite{neusel}, where she gives a classification result on arrangements consisting of $d$ lines and exactly one conic with some prescribed singularities admitting the so-called \textbf{tree resolution}. It turns out that some of our examples are included in her \textit{Bilderbuch}, so there is another mysterious connection between the freeness property and the property of having a tree resolution for curves.

In order to decide whether a certain weak combinatorics can be realized over the real or complex numbers, one can determine (presumably effective) numerical constraints, such as Hirzebruch-type inequalities. Here we present a general tool that we can apply to reduced plane curves with ordinary double, ordinary triple and ordinary quadruple points.
\begin{theorem}
\label{Hir}
Let $C \subset \mathbb{P}^{2}_{\mathbb{C}}$ be a reduced plane curve of degree $m\geq 6$ admitting only $n_{2}$ ordinary double, $n_{3}$ ordinary triple and $n_{4}$ ordinary quadruple points. Then one has
\begin{equation*}
9n_{2} + \frac{117}{4}n_{3} + 60n_{4} \leq 5m^{2}-6m.
\end{equation*}
\end{theorem}
In the paper we work, most of the time, over the complex numbers and our symbolic computations are preformed using \verb}SINGULAR} \cite{Singular}.

\section{Preliminaries}
Here we want to present preparatory tools that will be used extensively in our classification.

Let $S := \mathbb{C}[x,y,z]$ denote the coordinate ring of $\mathbb{P}^{2}_{\mathbb{C}}$, and for a homogeneous polynomial $f \in S$ let $J_{f}$ denote the Jacobian ideal associated with $f$, that is, the ideal of the form $J_{f} = \langle \partial_{x}\, f, \partial_{y} \, f, \partial_{z} \, f \rangle$.
\begin{definition}
Let $p$ be an isolated singularity of a polynomial $f\in \mathbb{C}[x,y]$. Since we can change the local coordinates, assume that $p=(0,0)$.
Furthermore, the number 
$$\mu_{p}=\dim_\mathbb{C}\left(\mathbb{C}\{x,y\} /\bigg\langle \partial_{x}\, f ,\partial_{y}\, f \bigg\rangle\right)$$
is called the Milnor number of $f$ at $p$.

The number
$$\tau_{p}=\dim_\mathbb{C}\left(\mathbb{C}\{x,y\}/\bigg\langle f, \partial_{x}\, f ,\partial_{y} \, f \bigg\rangle \right)$$
is called the Tjurina number of $f$ at $p$.
\end{definition}

For a projective situation, with a point $p\in \mathbb{P}^{2}_{\mathbb{C}}$ and a homogeneous polynomial $f\in \mathbb{C}[x,y,z]$, we take local affine coordinates such that $p=(0,0,1)$ and then the dehomogenization of $f$.

Finally, the total Tjurina number of a given reduced curve $C \subset \mathbb{P}^{2}_{\mathbb{C}}$ is defined as
$$\tau(C) = \sum_{p \in {\rm Sing}(C)} \tau_{p}.$$ 

Recall that a singularity is called quasi-homogeneous if and only if there exists a holomorphic change of variables so that the defining equation becomes weighted homogeneous. If $C = \{f=0\}$ is a reduced plane curve with only quasi-homogeneous singularities,  then by \cite[Satz]{KS} one has
$$\tau(C) = \sum_{p \in {\rm Sing}(C)} \tau_{p} = \sum_{p \in {\rm Sing}(C)} \mu_{p} = \mu(C),$$
which means that the total Tjurina number is equal to the total Milnor number of $C$.

Next, we will need an important invariant that is defined in the language of the syzygies of $J_{f}$.
\begin{definition}
Consider the graded $S$-module of Jacobian syzygies of $f$, namely $$AR(f)=\{(a,b,c)\in S^3 : a\partial_{x} \, f + b \partial_{y} \, f + c \partial_{z} \, f = 0 \}.$$
The minimal degree of non-trivial Jacobian relations for $f$ is defined to be 
$${\rm mdr}(f):=\min\{r : AR(f)_r\neq (0)\}.$$ 
\end{definition}
\begin{definition}
A reduced curve $C \subset \mathbb{P}^{2}_{\mathbb{C}}$ is free if the Jacobian ideal $J_{f}$ is saturated with respect to $\mathfrak{m} = \langle x,y,z\rangle$.
\end{definition}
It is somewhat difficult to check the freeness property using the above definition. However, it turns out that we can check the freeness property using the language of the minimal degree of (non-trivial) Jacobian relations and the total Tjurina numbers according to a result due to du Plessis and Wall \cite{duP}.
\begin{theorem}
\label{dddp}
Let $C = \{f=0\}$ be a reduced curve in $\mathbb{P}^{2}_{\mathbb{C}}$. One has
\begin{equation}
\label{duPles}
(d-1)^{2} - r(d-r-1) = \tau(C)
\end{equation}
if and only if $C = \{f=0\}$ is a free curve, and then $r \leq (d-1)/2$.
\end{theorem}
In order to perform our classification, we will need the following result \cite[Theorem 2.1]{DimcaSernesi}.
\begin{theorem}[Dimca-Sernesi]
\label{sern}
Let $C = \{f=0\}$ be a reduced curve of degree $d$ in $\mathbb{P}^{2}_{\mathbb{C}}$ having only quasi-homogeneous singularities. Then $${\rm mdr}(f) \geq \alpha_{C}\cdot d - 2,$$
where $\alpha_{C}$ denotes the Arnold exponent of $C$.
\end{theorem}
It is worth recalling that the Arnold exponent of a given reduced curve $C \subset \mathbb{P}^{2}_{\mathbb{C}}$ is defined as the minimum over all Arnold exponents of singular points $p$ in $C$. In modern language, the Arnold exponents of a singular point $p$ of a curve $C$ are nothing but the log canonical thresholds ${\rm lct}_{p}$ of the pair $(C,p)$. In the case of ordinary singularities, we have the following result \cite[Theorem 1.3]{Cheltsov}. 
\begin{theorem}
Let $C$ be a reduced curve in $\mathbb{C}^{2}$ which has degree $m$ and let $p \in {\rm Sing}(C)$. Then ${\rm lct}_{p}(C) \geq \frac{2}{m}$, and the equality holds if and only if $C$ is a union of $m$ lines passing through $p$.
\end{theorem}
\begin{remark}
\label{lct}
If $p=(0,0) \in \mathbb{C}^{2}$ is an ordinary singularity of multiplicity $r$ determined by $C = \{ f=0\}$, then ${\rm lct}_{p}(f) = \frac{2}{r}$.
\end{remark}
\section{Partial Classification}
Our classification procedure is based on the following general approach which uses combinatorial constraints derived from the combinatorics of free curves.

If $\mathcal{CL} \subset \mathbb{P}^{2}_{\mathbb{C}}$ is an arrangement of $d\geq 3$ lines and one conic having $n_{2}$ nodes, $n_{3}$ ordinary triple and $n_{4}$ ordinary quadruple points, then by B\'ezout's theorem we have
\begin{equation}
    2d + \binom{d}{2} = n_{2} + 3n_{3} + 6n_{4},
\end{equation}
and this is what we call the naive count. This simple combinatorial count is crucial because it allows us to determine which weak combinatorics are admissible for our conic-line arrangements. For instance, there does not exist an arrangement of two lines and one conics with $6$ ordinary double intersections. However, this combinatorial count is not strong enough for our purposes and we need to find additional restrictions. In this sense, we introduce the second constraint which concerns the total Tjurina number of $\mathcal{CL} = \{f=0\}$ with degree $d+2$ and $r:={\rm mdr}(f)$, namely
\begin{equation}
r^{2} - r(d+1) + (d+1)^2 = \tau(\mathcal{CL}) = n_{2} + 4n_{3} + 9n_{4},
\end{equation}
where the left-hand side follows from the fact that for an ordinary singularity $p \in {\rm Sing}(C)$ of multiplicity $m_{p} < 5$ one has $\tau_{p} = \mu_{p} = (m_{p}-1)^2$, see for instance \cite[Section 2]{KP}. The above equation decodes the fact that if we can find a conic-line arrangement with ordinary double, ordinary triple, and ordinary quadruple intersection points with suitable $r$, then our arrangement is free. Based on this discussion, our last missing element is to find numerical constraints on $r={\rm mdr}(f)$.
We have the following result.
\begin{proposition}
Let $\mathcal{CL} = \{f=0\}$ be an arrangement of $d\geq 3$ lines and one smooth conic that admits only ordinary singularities of multiplicity $<5$. Assume that $\mathcal{CL}$ is free, then
$${\rm mdr}(f) \in \bigg\{\bigg\lceil \frac{d-2}{2} \bigg\rceil, \bigg\lfloor \frac{d+1}{2} \bigg \rfloor \bigg\}.$$
\end{proposition}
\begin{proof}
Since $\mathcal{CL}$ is free, then by Theorem \ref{dddp} one has ${\rm mdr}(f) \leq \lfloor \frac{d+1}{2} \rfloor$. Moreover, if $\mathcal{CL}$ admits only ordinary singularities with multiplicities $<5$, then $\alpha_{\mathcal{CL}} = \frac{1}{2}$, and this follows from Remark \ref{lct}. This gives us that
$${\rm mdr}(f) \geq \frac{d+2}{2} - 2 = \frac{d-2}{2},$$
and we finally obtain
$${\rm mdr}(f)\geq \bigg\lceil \frac{d-2}{2} \bigg\rceil.$$
\end{proof}
\noindent
For $d\geq 3$ and 
\begin{equation}
\label{mdrr}
r:=r(d) \in \bigg\{\bigg\lceil \frac{d-2}{2} \bigg\rceil, \bigg\lfloor \frac{d+1}{2} \bigg \rfloor \bigg\},
\end{equation}
we consider the following Diophantine system of equations:
\begin{equation}
\label{diophant}
\begin{cases}
    r^{2} - r(d+1) + (d+1)^2  = n_{2} + 4n_{3} + 9n_{4}, \\
    2d + \frac{d(d-1)}{2} = n_{2} + 3n_{3} + 6n_{4}.
\end{cases}
\end{equation}
Now our goal is to find all non-negative integer solutions $(n_{2},n_{3},n_{4})$ to \eqref{diophant}, depending on $d$ and $r=r(d)$, to get a complete weak combinatorial description of the expected free conic-line arrangements. The last step boils down to deciding on the existence/non-existence of a geometric realization of a given weak combinatorics, which is a completely non-trivial problem. It is worth emphasizing here that if we can find an arrangement $\mathcal{CL} = \{f=0\}$ that satisfies conditions \eqref{diophant} with $r={\rm mdr}(f)$ that satisfies \eqref{mdrr}, then $\mathcal{CL}$ is automatically free. In our classification we use a well-known (even folkloric) result which tells us that all loopless matroids of rank $3$ with up to $6$ elements are vector matroids which can be represented geometrically as line arrangements in $\mathbb{P}^{2}_{\mathbb{K}}$ with $\mathbb{K}$ being any infinite field.

Now we are read to deliver a proof of Theorem \ref{class}.
\begin{proof}
Our proof is a degree-wise classification.
\begin{enumerate}
    \item[($d=3$):] Solving system \eqref{diophant} with $r \in \{1,2\}$, we obtain exactly three admissible weak combinatorics, namely
    $$(n_{2},n_{3},n_{4}) \in \{(3,0,1), (0,3,0), (0,1,1)\}.$$
    We start with a geometric realization of the first possibility. 
    Consider the arrangement $\mathcal{CL}_{1}$ given by the following polynomial
    $$Q_{1}(x,y,z) = x(x^2 + y^2 - z^2)(y-x-z)(y+x-z).$$ Since $\tau(\mathcal{CL}_{1})=12$, because $(n_{2},n_{3},n_{4}) = (3,0,1)$, and ${\rm mdr}(Q_{1}) = 2$, hence both \eqref{mdrr} and \eqref{diophant} are satisfied.

    Let us now consider the following arrangement $\mathcal{CL}_{2}$ given by
    $$Q_{2} = y(x^2 + y^2 -16z^2 )(x+y-4z)(x-y+4z).$$ Since $(n_{2},n_{3},n_{4}) = (0,3,0)$, $\tau(\mathcal{CL}_{2})=12$ and ${\rm mdr}(Q_{2}) = 2$, hence both \eqref{mdrr} and \eqref{diophant} are satisfied.

    Observe that the weak combinatiorics $(n_{2},n_{3},n_{4})=(0,1,1)$ cannot be realized geometrically. If such an arrangement existed, we would be able to find two lines in the arrangement that intersect at two different points.
    \item[($d=4$):] 
    Solving system \eqref{diophant} with $r \in \{1,2\}$, we obtain exactly one possible weak combinatorics, namely $(n_{2},n_{3},n_{4}) = (2,2,1)$. 
    Consider the arrangement $\mathcal{CL}_{3}$ given by the following polynomial
    $$Q_{3}(x,y,z) = xy(x^2 + y^2 - z^2)(y-x-z)(y+x-z).$$ Observe that $\tau(\mathcal{CL}_{3})=19$, since we have $(n_{2},n_{3},n_{4}) = (2,2,1)$, and ${\rm mdr}(Q_{3}) = 2$, so both \eqref{mdrr} and \eqref{diophant} are satisfied.
    \item[($d=5$):] Solving system \eqref{diophant} with $r \in \{2,3\}$, we obtain exactly three possible weak combinatorics, namely 
    $$(n_{2},n_{3},n_{4}) \in \{(2,2,2),(5,1,2),(2,4,1)\}.$$ 
    We start with the first weak combinatorics.  Consider the arrangement $\mathcal{CL}_{4}$ given by the following polynomial
    $$Q_{4}(x,y,z) = y(x^2 + y^2 - z^2)(y-x-z)(y+x-z)(-y-x-z)(-y+x-z).$$ Since $\tau(\mathcal{CL}_{4})=28$, because we get the required intersections, and ${\rm mdr}(Q_{4}) = 2$, hence both \eqref{mdrr} and \eqref{diophant} are satisfied.
    
    For the second weak combinatorics, consider the arrangement $\mathcal{CL}_{5}$ given by the following polynomial
    $$Q_{5}(x,y,z) = y(x^2 + y^2 - z^2)(y-x-z)(y+x-z)(y+2x+2z)(y-2x+2z).$$ 
    Observe that $\tau(\mathcal{CL}_{5})=27$ and ${\rm mdr}(Q_{5}) = 3$, hence both \eqref{mdrr} and \eqref{diophant} are satisfied
    
    Now we will show that the combinatorics $(n_{2},n_{3},n_{4}) = (2,4,1)$ cannot be realized geometrically. Note that if such an arrangement existed, then we would have a subarrangement with the property that one quadruple point, two triple points, one double point are located on a given conic, and one additional double intersection point is located away from the conic, and these are all intersections between our curves. To get two more triple intersections, we have to draw a line through the two double points, and we end up with a contradiction because we have two lines that intersect at two different points. 
    \item[($d=6$):] Solving system \eqref{diophant} with $r \in \{2,3\}$, we obtain exactly five possible weak combinatorics, namely 
    $$(n_{2},n_{3},n_{4}) \in \{(3,0,4),(3,4,2),(6,1,3),(0,3,3),(0,7,1)\}.$$ 
        
    Let us start with the first weak combinatorics, namely $(n_{2},n_{3},n_{4}) = (3,0,4)$. Consider the arrangement $\mathcal{CL}_{6}$ given by
    $$Q_{6}(x,y,z) =xy(x^2 + y^2 - z^2)(y+x-z)(y-x-z)(y+x+z)(y-x+z).$$
    Observe that $\tau(\mathcal{CL}_{6})=39$, since $(n_{2},n_{3},n_{4}) = (3,0,4)$, and ${\rm mdr}(Q_{6}) = 2$, hence both \eqref{mdrr} and \eqref{diophant} are satisfied.

    Let us focus on the second weak combinatorics. Consider the arrangement $\mathcal{CL}_{7}$ given by
     $$Q_{7}(x,y,z) =x(x^2 + y^2 - z^2)(y-x)(y+x)\bigg(x-\frac{\sqrt{2}}{2}z\bigg)\bigg(x+\frac{\sqrt{2}}{2}z\bigg)\bigg(y+\frac{\sqrt{2}}{2}z\bigg).$$
    Observe that $\tau(\mathcal{CL}_{7})=37$, since $(n_{2},n_{3},n_{4}) = (3,4,2)$, and ${\rm mdr}(Q_{7}) = 3$, hence both \eqref{mdrr} and \eqref{diophant} are satisfied.
    
    Finally, let us consider the third weak combinatorics. Consider the arrangement $\mathcal{CL}_{8}$ given by
    $$Q_{8}(x,y,z) = (x-z)(x+z)(y-z)(y+z)(y+x)(y-x)(-2x^2 -2y^2 + 3z^2 +xy -xz +yz).$$
    Observe that $\tau(\mathcal{CL}_{8})=37$, since $(n_{2},n_{3}, n_{4}) = (6,1,3)$, and ${\rm mdr}(Q_{8}) = 3$, hence both \eqref{mdrr} and \eqref{diophant} are satisfied.

    To complete our classification for $d=6$, we need to show that the weak combinatorics $(n_{2},n_{3},n_{4})\in \{(0,3,3), (0,7,1)\}$ cannot be realized geometrically over the reals. In order to verify this claim we can start by looking at real line arrangements with only double, triple and quadruple points, and the following weak combinatorics are admissible:
    $$(n_{2}^{\mathcal{L}}, n_{3}^{\mathcal{L}}, n_{4}^{\mathcal{L}}) \in \{(15,0,0), (12,1,0), (9,2,0),(3,4,0), (9,0,1), (6,1,1)\},$$
    where $n_{i}^{\mathcal{L}}$ denotes the number of $i$-fold points of a given line arrangement $\mathcal{L} \subset \mathbb{P}^{2}_{\mathbb{R}}$.
    If we add a smooth conic to a given line arrangement, then by B\'ezout's Theorem the number of singularities of the resulting conic-line arrangement can either remain the same or increase. This implies that we can restrict our attention to the following weak combinatorics:
    $$(n_{2}^{\mathcal{L}}, n_{3}^{\mathcal{L}}, n_{4}^{\mathcal{L}}) \in \{(3,4,0), (6,1,1)\}.$$
    Furthermore, it follows from the above argument that we cannot construct a conic-line arrangement with the weak-combinatorics $(n_{2},n_{3},n_{4}) = (0,3,3)$ and we are left with the case $(n_{2},n_{3},n_{4}) = (0,7,1)$. 
    If we start with $\mathcal{L}$ such that  $(n_{2}^{\mathcal{L}}, n_{3}^{\mathcal{L}}, n_{4}^{\mathcal{L}}) = (3,4,0)$, then we have to add a conic that passes through one triple point and all double intersection points, but in that way we produce a conic-line arrangement with the weak combinatorics $(n_{2},n_{3},n_{4}) = (3,6,1)$. Now we look at $(n_{2}^{\mathcal{L}}, n_{3}^{\mathcal{L}}, n_{4}^{\mathcal{L}}) = (6,1,1)$. This arrangement is constructed in the following way. We take $3$ lines passing through a point $P\in \mathbb{P}^{2}_{\mathbb{R}}$, and $2$ lines passing through a point $Q \in \mathbb{P}^{2}_{\mathbb{R}}$ such that $P\neq Q$. Then we join these two pencils by the line passing through $P$ and $Q$. To construct our conic-line arrangement, we have to add a conic that would pass through $6$ double points of $\mathcal{L}$, but such a conic cannot exist by B\'ezout's Theorem since in the line arrangement we can find $3$ collinear double points.

     \item[($d=7$):] Solving system \eqref{diophant} with $r \in \{3,4\}$, we obtain exactly five possible weak combinatorics, namely 
    $$(n_{2},n_{3},n_{4}) \in \{(5,2,4),(5,4,3),(8,1,4),(2,5,3),(2,7,2)\}.$$ 
     For the first weak combinatorics, consider the arrangement $\mathcal{CL}_{9}$ given by
    $$Q_{9}(x,y,z) = x(x-z)(x+z)(y-z)(y+z)(y-x)(y+x)(x^2 + y^2 -2z^2 ).$$
    We can easily observe that $\tau(\mathcal{CL}_{9}) = 49$ and ${\rm mdr}(Q_{9})=3$, hence both \eqref{mdrr} and \eqref{diophant} are satisfied.

    Let us now pass to the second weak combinatorics and consider the arrangement $\mathcal{CL}_{10}$ given by
    \begin{multline*}
        Q_{10}(x,y,z) = x(-4x^2 + 12y^2 - 4yz - 5z^2 )\bigg(y+x+\frac{1}{2}z\bigg)\bigg(y-x+\frac{1}{2}z\bigg) \\ \bigg(y-\frac{3}{4}x\bigg)\bigg(y+\frac{3}{4}x\bigg)\bigg(y-2x+ \frac{5}{2}z\bigg)\bigg(y+2x+\frac{5}{2}z\bigg).
    \end{multline*} 
    One can easily check that $\tau(\mathcal{CL}_{10})=48$ and ${\rm mdr}(Q_{10})=4$, hence both \eqref{mdrr} and \eqref{diophant} are satisfied.
    
    Finally, let us pass to the third weak combinatorics. Consider the arrangement $\mathcal{CL}_{11}$ given by
   \begin{multline*}
     Q_{11}(x,y,z) = x(y+x+z)(y-x+z)(3x^2 + 5y^2 - 6yz - 11z^2 ) \\ \bigg(y - \frac{6}{10} x- \frac{22}{10}z \bigg)\bigg(y + \frac{6}{10}x - \frac{22}{10}z\bigg)\bigg(y - \frac{6}{10}x + \frac{2}{10}z \bigg)\bigg(y + \frac{6}{10} x + \frac{2}{10}z\bigg).
   \end{multline*}
   One can check that $\tau(\mathcal{CL}_{11})=48$ and ${\rm mdr}(Q_{10})=4$, so both \eqref{mdrr} and \eqref{diophant} are satisfied.

   \item[($d=8$):] Solving system \eqref{diophant} with $r \in \{3,4\}$, we obtain exactly five possible weak combinatorics, namely 
    $$(n_{2},n_{3},n_{4}) \in \{(5,5,4),(2,8,3),(8,2,5),(2,4,5),(5,1,6)\}.$$ 
    Let us focus on the first weak combinatorics. Consider the arrangement $\mathcal{CL}_{12}$ given by the equation
    $$Q_{12}(x,y,z) = xz(x+z)(x-z)(y+x-2z)(y+x)(y-x)(y+x+2z)(3x^{2}+y^{2}-4z^{2}).$$
    We can check that $(n_{2}, n_{3}, n_{4}) = (5,5,4)$, so we have $\tau(\mathcal{CL}_{12})=61$, and ${\rm mdr}(Q_{12})=4$, so both \eqref{mdrr} and \eqref{diophant} are satisfied.

    Let us go to the second weak combinatorics and we consider the arrangement $\mathcal{CL}_{13}$ given by
    \begin{multline*}
    Q_{13}(x,y,z) = xz(x+z)(x-z)(-3x^{2}+4y^{2}-z^{2})\bigg(y+\frac{1}{2}x+\frac{1}{2}z\bigg)\bigg(y+\frac{1}{2}x-\frac{1}{2}z\bigg) \\
    \bigg(y-\frac{1}{2}x-\frac{1}{2}z\bigg)\bigg(y-\frac{1}{2}x+\frac{1}{2}z\bigg).
    \end{multline*}
    We can check that $(n_{2}, n_{3}, n_{4}) = (2,8,3)$, so we have $\tau(\mathcal{CL}_{13})=61$, and ${\rm mdr}(Q_{13})=4$, so both \eqref{mdrr} and \eqref{diophant} are satisfied.
    
    Finally, we look at the third weak combinatorics. Consider the arrangement $\mathcal{CL}_{14}$ given by 
    $$Q_{14}(x,y,z) = xy(x+z)(x-z)(y-z)(y+z)(y-x)(y+x)(3x^{2}+y^{2}-4z^{2}).$$ We can check that $(n_{2}, n_{3}, n_{4}) = (8,2,5)$, so we have $\tau(\mathcal{CL}_{14})=61$, and ${\rm mdr}(Q_{14})=4$, so both \eqref{mdrr} and \eqref{diophant} are satisfied.
     \item[($d=9$):] Solving system \eqref{diophant} with $r \in \{4,5\}$, we obtain exactly nine possible weak combinatorics, namely 
     \begin{multline*}
        (n_{2},n_{3},n_{4}) \in \{(0,10,4),(3,7,5),(6,4,6),(9,1,7),(0,12,3),(3,9,4),\\ (6,6,5),(9,3,6),(12,0,7) \}.
     \end{multline*}
Here we are able to construct just one weak combinatorics, namely $(n_{2},n_{3},n_{4}) = (6,4,6)$. Consider the arrangement $\mathcal{CL}_{15}$ given by
\begin{multline*}
Q_{15}(x,y,z) = xy(x-z)(x+z)(y+z)(y-z)(y-x-z)(y-x+z)(y-x) \\ (-x^{2}+xy-y^{2}+z^{2}).
\end{multline*}
Since $\tau(\mathcal{CL}_{15})=76$ and ${\rm mdr}(\mathcal{CL}_{15})=4$, hence both \eqref{mdrr} and \eqref{diophant} are satisfied.
\item[($d=10$):] Solving system \eqref{diophant} with $r \in \{4,5\}$, we obtain exactly seven possible weak combinatorics, namely 
     \begin{multline*}
        (n_{2},n_{3},n_{4}) \in \{(2,11,5),(5,8,6),(8,5,7),(11,2,8),(2,7,7),(5,4,8),(8,1,9)\}.
        \end{multline*}
        We are going to show that the weak combinatorics $(n_{2},n_{3},n_{4})=(8,1,9)$ can be realized geometrically. Consider the arrangement $\mathcal{CL}_{16}$ given by 
\begin{multline*}
    Q_{16}(x,y,z) = xyz(x-z)(x+z)(y+z)(y-z)(y-x-z)(y-x+z)(y-x) \\ (-x^{2}+xy-y^{2}+z^{2}).
\end{multline*}
Since
$\tau(\mathcal{CL}_{16})=93$ and ${\rm mdr}(Q_{16})=4$, so both \eqref{mdrr} and \eqref{diophant} are satisfied.
\end{enumerate}
\end{proof}
\begin{remark}
Looking at a classification result of Neusel \cite{neusel}, one can notice that the weak combinatorics $(1,7;2,7,2)$ can be realized geometrically, presumably over the real numbers, or at least her \textit{Bilderbuch} may suggest this. Because of this ambiguity, we will briefly explain here why this is not the case by reproducing her picture by equations. The starting point is the arrangement of seven lines $\mathcal{L}$ given by
$$Q(x,y,z)=xy(x-z)(x+z)(y-z)(y+z)(y-x).$$
Consider the intersection points 
$$P_{1} = (-1:1:1),\quad  P_{2} = (-1:0:1), \quad P_{3}=(0:0:1), \quad P_{4}=(1:1:1), $$
$$P_{5}=(0:-1:1), \quad P_{6}=(1:-1:1).$$
Then, according to what we can see in the picture, the points $P_{1},P_{2},P_{3},P_{4},P_{5},P_{6}$ should be contained in a smooth conic. However, as a simple calculation shows, this is not the case.
\end{remark}
\begin{remark}
As we can see, in the cases $d=7,8,9,10$ some weak combinatorics remained untouched by us since we could not find their geometric realization over the real numbers. It is a separate and quite difficult problem to decide whether these weak combinatorics can be realized over the real or the complex numbers.
\end{remark}
\section{Graphical realizations of some arrangements}
In this section we want to present some geometric realizations of the arrangements constructed in the previous section to give a sense of their symmetry.

\textbf{Arrangement 1}.
\vspace{1cm}
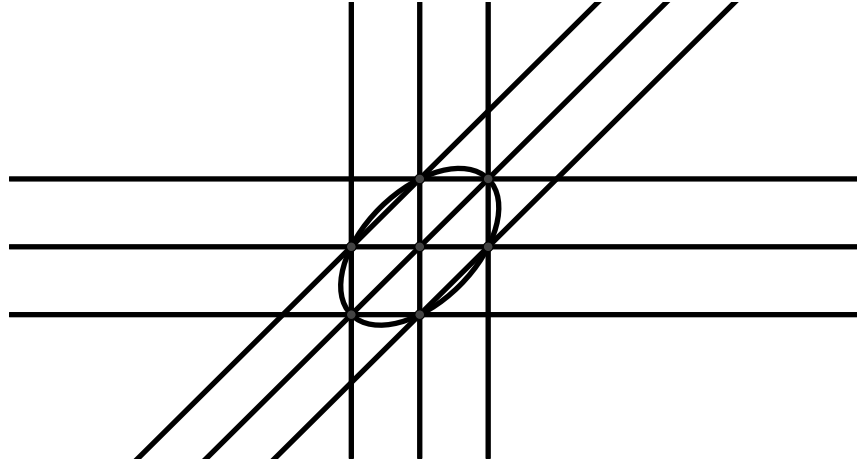
\begin{figure}[h]
\definecolor{uuuuuu}{rgb}{0.26666666666666666,0.26666666666666666,0.26666666666666666}
\centering
\begin{tikzpicture}[line cap=round,line join=round,>=triangle 45,x=1.0cm,y=1.0cm,scale=0.9]
\clip(-5.993665295370997,-3.112645222782914) rectangle (6.387229328409073,3.602513463852773);
\draw [line width=2.pt] (0.,-3.112645222782914) -- (0.,3.602513463852773);
\draw [line width=2.pt] (1.,-3.112645222782914) -- (1.,3.602513463852773);
\draw [line width=2.pt] (-1.,-3.112645222782914) -- (-1.,3.602513463852773);
\draw [line width=2.pt,domain=-5.993665295370997:6.387229328409073] plot(\x,{(-0.-0.*\x)/1.});
\draw [line width=2.pt,domain=-5.993665295370997:6.387229328409073] plot(\x,{(--1.-0.*\x)/1.});
\draw [line width=2.pt,domain=-5.993665295370997:6.387229328409073] plot(\x,{(-1.-0.*\x)/1.});
\draw [line width=2.pt,domain=-5.993665295370997:6.387229328409073] plot(\x,{(--1.--1.*\x)/1.});
\draw [line width=2.pt,domain=-5.993665295370997:6.387229328409073] plot(\x,{(-0.--1.*\x)/1.});
\draw [line width=2.pt,domain=-5.993665295370997:6.387229328409073] plot(\x,{(-1.--1.*\x)/1.});
\draw [rotate around={-135.:(0.,0.)},line width=2.pt] (0.,0.) ellipse (1.4142135623730951cm and 0.816496580927726cm);
\begin{scriptsize}
\draw [fill=uuuuuu] (-1.,0.) circle (2.0pt);
\draw [fill=uuuuuu] (0.,1.) circle (2.0pt);
\draw [fill=uuuuuu] (1.,1.) circle (2.0pt);
\draw [fill=uuuuuu] (1.,0.) circle (2.0pt);
\draw [fill=uuuuuu] (0.,0.) circle (2.0pt);
\draw [fill=uuuuuu] (-1.,-1.) circle (2.0pt);
\draw [fill=uuuuuu] (0.,-1.) circle (2.0pt);
\end{scriptsize}
\end{tikzpicture}
\caption{Arrangement of type $(9;6,4,6)$.}

\end{figure}

\textbf{Arrangement 2}.
\vspace{1cm}
\begin{figure}[ht]
\centering
\definecolor{wqwqwq}{rgb}{0.3764705882352941,0.3764705882352941,0.3764705882352941}
\centering
\begin{tikzpicture}[line cap=round,line join=round,>=triangle 45,x=1.0cm,y=1.0cm,scale=0.9]
\clip(-7.020381482257806,-2.5838330898658195) rectangle (6.843081857365727,4.12762778149371);
\draw [line width=2.pt] (0.,-2.5838330898658195) -- (0.,4.12762778149371);
\draw [line width=2.pt,domain=-7.020381482257806:6.843081857365727] plot(\x,{(-0.-0.*\x)/1.});
\draw [line width=2.pt] (-1.,-2.5838330898658195) -- (-1.,4.12762778149371);
\draw [line width=2.pt] (1.,-2.5838330898658195) -- (1.,4.12762778149371);
\draw [line width=2.pt,domain=-7.020381482257806:6.843081857365727] plot(\x,{(--1.-0.*\x)/1.});
\draw [line width=2.pt,domain=-7.020381482257806:6.843081857365727] plot(\x,{(-1.-0.*\x)/1.});
\draw [line width=2.pt,domain=-7.020381482257806:6.843081857365727] plot(\x,{(-0.--1.*\x)/1.});
\draw [line width=2.pt,domain=-7.020381482257806:6.843081857365727] plot(\x,{(-0.-1.*\x)/1.});
\draw [rotate around={90.:(0.,0.)},line width=2.pt] (0.,0.) ellipse (2.cm and 1.1547005383792515cm);
\begin{scriptsize}
\draw [fill=wqwqwq] (-1.,1.) circle (2.5pt);
\draw [fill=wqwqwq] (1.,1.) circle (2.5pt);
\draw [fill=wqwqwq] (0.,0.) circle (2.5pt);
\draw [fill=wqwqwq] (1.,-1.) circle (2.5pt);
\draw [fill=wqwqwq] (-1.,-1.) circle (2.5pt);
\end{scriptsize}
\end{tikzpicture}
\caption{Arrangement of type $(8;8,2,5)$.}
\end{figure}
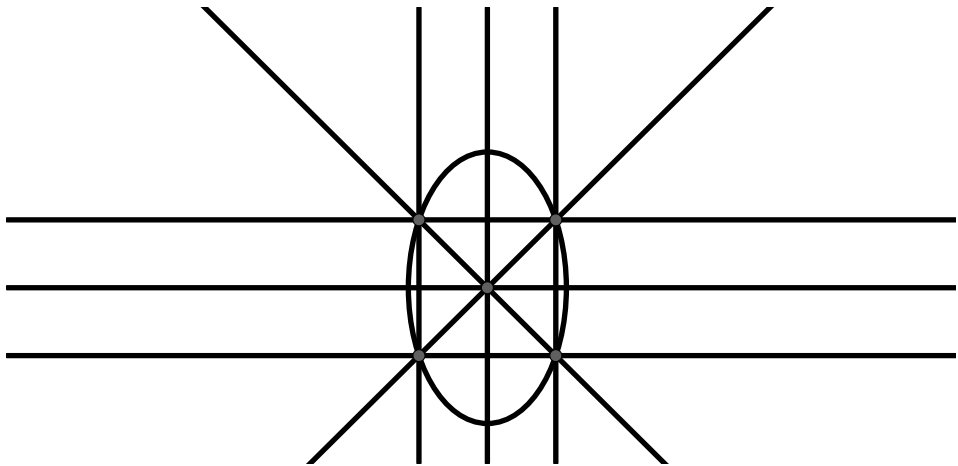
\newpage
\textbf{Arrangement 3.}
\vspace{1cm}
\begin{figure}[h]
\definecolor{wqwqwq}{rgb}{0.3764705882352941,0.3764705882352941,0.3764705882352941}
\centering
\begin{tikzpicture}[line cap=round,line join=round,>=triangle 45,x=1.0cm,y=1.0cm,scale=0.8]
\clip(-8.626789898750246,-1.5957639615131654) rectangle (8.312131702183478,7.148303929753834);
\draw [line width=2.pt] (0.,-1.5957639615131654) -- (0.,7.148303929753834);
\draw [rotate around={0.:(0.,0.6)},line width=2.pt] (0.,0.6) ellipse (2.065591117977289cm and 1.6cm);
\draw [line width=2.pt,domain=-8.626789898750246:8.312131702183478] plot(\x,{(--2.2--0.6*\x)/1.});
\draw [line width=2.pt,domain=-8.626789898750246:8.312131702183478] plot(\x,{(--2.2-0.6*\x)/1.});
\draw [line width=2.pt,domain=-8.626789898750246:8.312131702183478] plot(\x,{(-1.-1.*\x)/1.});
\draw [line width=2.pt,domain=-8.626789898750246:8.312131702183478] plot(\x,{(-1.--1.*\x)/1.});
\draw [line width=2.pt,domain=-8.626789898750246:8.312131702183478] plot(\x,{(-0.2--0.6*\x)/1.});
\draw [line width=2.pt,domain=-8.626789898750246:8.312131702183478] plot(\x,{(-0.2-0.6*\x)/1.});
\begin{scriptsize}
\draw [fill=wqwqwq] (2.,1.) circle (2.5pt);
\draw [fill=wqwqwq] (0.,-0.2) circle (2.5pt);
\draw [fill=wqwqwq] (-2.,1.) circle (2.5pt);
\draw [fill=wqwqwq] (0.,2.2) circle (2.5pt);
\draw [fill=wqwqwq] (0.,-1.) circle (2.5pt);
\end{scriptsize}
\end{tikzpicture}
\caption{Arrangement of type $(7;8,1,4)$.}
\end{figure}
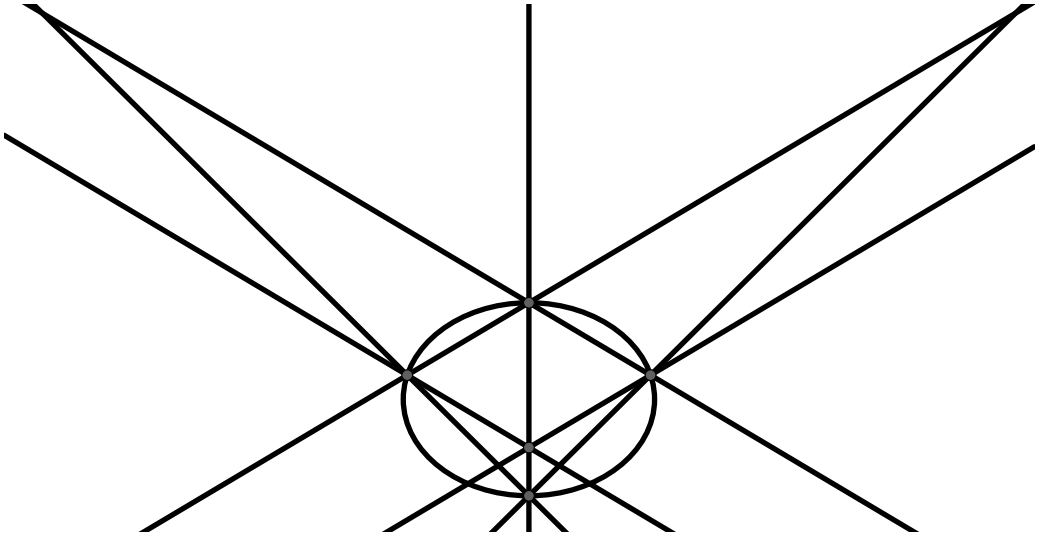

\textbf{Arrangement 4.}
\vspace{1cm}
\begin{figure}[ht]
\definecolor{wqwqwq}{rgb}{0.3764705882352941,0.3764705882352941,0.3764705882352941}
\centering
\begin{tikzpicture}[line cap=round,line join=round,>=triangle 45,x=1.0cm,y=1.0cm,scale=0.9]
\clip(-7.874185655657578,-2.778995909572207) rectangle (7.819098782371612,4.81830224520538);
\draw [line width=2.pt] (0.,-2.778995909572207) -- (0.,4.81830224520538);
\draw [samples=50,domain=-0.99:0.99,rotate around={90.:(0.,0.16666666666666666)},xshift=0.cm,yshift=0.16666666666666666cm,line width=2.pt] plot ({0.6666666666666666*(1+(\x)^2)/(1-(\x)^2)},{1.1547005383792515*2*(\x)/(1-(\x)^2)});
\draw [samples=50,domain=-0.99:0.99,rotate around={90.:(0.,0.16666666666666666)},xshift=0.cm,yshift=0.16666666666666666cm,line width=2.pt] plot ({0.6666666666666666*(-1-(\x)^2)/(1-(\x)^2)},{1.1547005383792515*(-2)*(\x)/(1-(\x)^2)});
\draw [line width=2.pt,domain=-7.874185655657578:7.819098782371612] plot(\x,{(-0.5-1.*\x)/1.});
\draw [line width=2.pt,domain=-7.874185655657578:7.819098782371612] plot(\x,{(-0.5--1.*\x)/1.});
\draw [line width=2.pt,domain=-7.874185655657578:7.819098782371612] plot(\x,{(-0.--0.75*\x)/1.});
\draw [line width=2.pt,domain=-7.874185655657578:7.819098782371612] plot(\x,{(-0.-0.75*\x)/1.});
\draw [line width=2.pt,domain=-7.874185655657578:7.819098782371612] plot(\x,{(-2.5--2.*\x)/1.});
\draw [line width=2.pt,domain=-7.874185655657578:7.819098782371612] plot(\x,{(-2.5-2.*\x)/1.});
\begin{scriptsize}
\draw [fill=wqwqwq] (-2.,1.5) circle (2.5pt);
\draw [fill=wqwqwq] (0.,-0.5) circle (2.5pt);
\draw [fill=wqwqwq] (0.,-2.5) circle (2.5pt);
\draw [fill=wqwqwq] (2.,1.5) circle (2.0pt);
\draw [fill=wqwqwq] (-0.9090909090909091,-0.6818181818181818) circle (2.5pt);
\draw [fill=wqwqwq] (0.9090909090909091,-0.6818181818181818) circle (2.5pt);
\draw [fill=wqwqwq] (0.,0.) circle (2.5pt);
\end{scriptsize}
\end{tikzpicture}
\caption{Arrangement of type $(7;5,4,3)$.}
\end{figure}
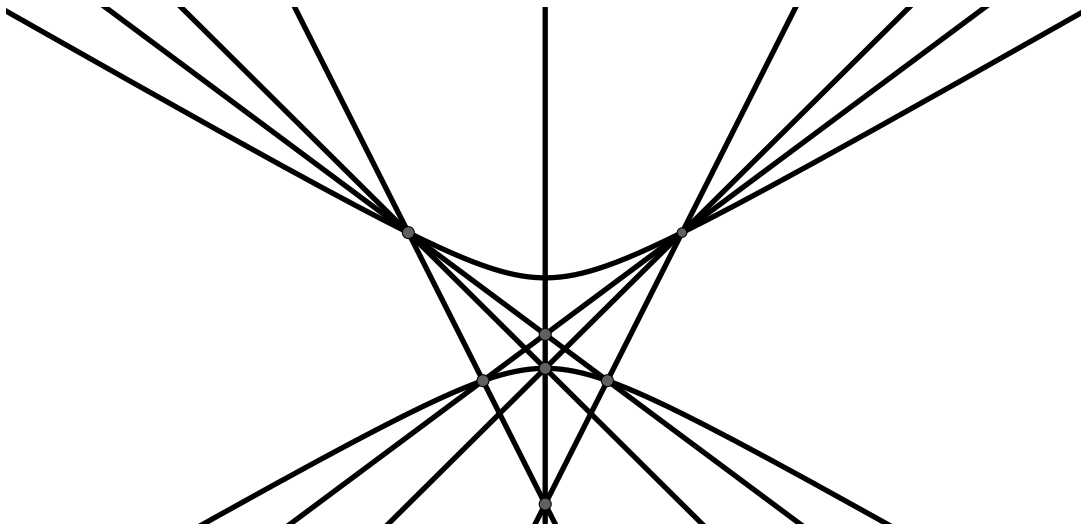
\newpage

\textbf{Arrangement 5.}
\vspace{1cm}
\begin{figure}[h]
\definecolor{yqyqyq}{rgb}{0.5019607843137255,0.5019607843137255,0.5019607843137255}
\definecolor{wqwqwq}{rgb}{0.3764705882352941,0.3764705882352941,0.3764705882352941}
\centering
\begin{tikzpicture}[line cap=round,line join=round,>=triangle 45,x=1.0cm,y=1.0cm]
\clip(-7.361199998992579,-3.2212539171466634) rectangle (6.856141156665902,3.6615232449100428);
\draw [line width=2.pt] (0.,0.) circle (1.cm);
\draw [line width=2.pt] (0.,0.) circle (1.0177943973136974cm);
\draw [line width=2.pt,domain=-7.361199998992579:6.856141156665902] plot(\x,{(--1.0359054352031525--1.0177943973136971*\x)/1.0177943973136974});
\draw [line width=2.pt,domain=-7.361199998992579:6.856141156665902] plot(\x,{(--1.0359054352031525-1.0177943973136974*\x)/1.0177943973136974});
\draw [line width=2.pt,domain=-7.361199998992579:6.856141156665902] plot(\x,{(--1.0359054352031525-1.0177943973136974*\x)/-1.0177943973136974});
\draw [line width=2.pt,domain=-7.361199998992579:6.856141156665902] plot(\x,{(--1.0359054352031525--1.0177943973136976*\x)/-1.0177943973136974});
\draw [line width=2.pt,domain=-7.361199998992579:6.856141156665902] plot(\x,{(-0.-0.*\x)/2.0355887946273947});
\draw [line width=2.pt] (0.,-3.2212539171466634) -- (0.,3.6615232449100428);
\begin{scriptsize}
\draw [fill=wqwqwq] (-1.0177943973136974,0.) circle (2.5pt);
\draw [fill=wqwqwq] (0.,1.0177943973136974) circle (2.5pt);
\draw [fill=yqyqyq] (1.0177943973136974,0.) circle (2.5pt);
\draw [fill=wqwqwq] (0.,-1.0177943973136974) circle (2.5pt);
\end{scriptsize}
\end{tikzpicture}
\caption{Arrangement of type $(6;3,0,4)$.}
\end{figure}
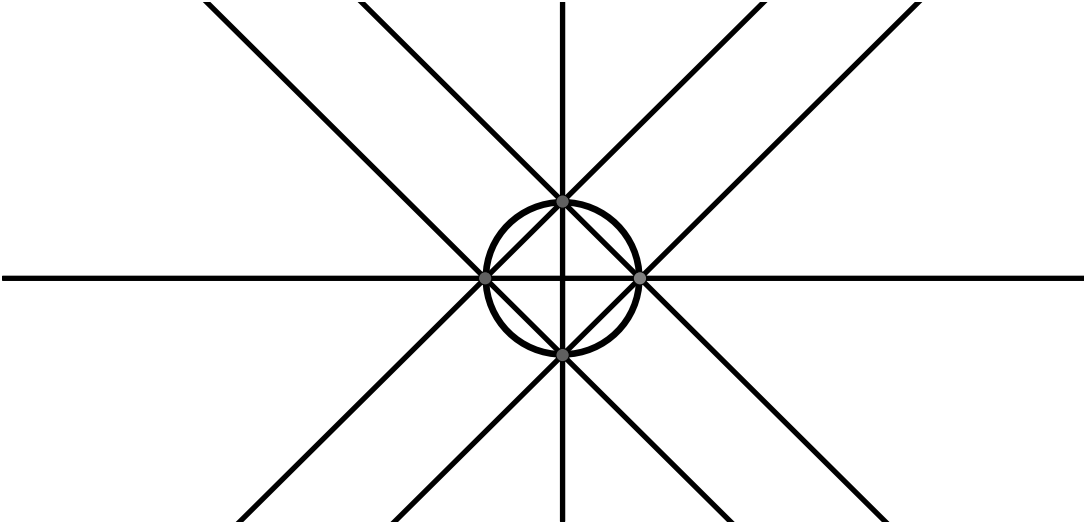

\section{Combinatorial constraints on plane curves with ordinary quasi-homogeneous singularities}
In this section we present our proof of Theorem \ref{Hir}.
\begin{proof}
We present a detailed outline of our proof since it is similar to \cite[Theorem B]{Pokora}. By the assumption $m= {\rm deg} \, C \geq 6$, so we can work with the pair $(\mathbb{P}^{2}_{\mathbb{C}}, \frac{1}{2}C)$ that is effective, and thus we can use an orbifold Bogomolov-Miyaoka-Yau inequality \cite[Section 11]{Langer}, namely
$$(\star) \, : \quad \sum_{p \in {\rm Sing}(C)} 3\bigg( \frac{1}{2}(\mu_{p}-1)+1-e_{orb}\bigg(p;\mathbb{P}^{2}_{\mathbb{C}}, \frac{1}{2} C\bigg)\bigg)\leq  \frac{5}{4}m^{2}- \frac{3}{2}m,$$
where $e_{orb}(p;\mathbb{P}^{2}_{\mathbb{C}},\alpha C)$ is the local orbifold Euler number of a given singularity $p \in {\rm Sing}(C)$ \cite[Definition 3.1]{Langer}, and $\mu_{p}$ denotes the Milnor number of $p$. Let us recall these local orbifold Euler numbers for our selection of singularities. Using \cite[Theorem 8.7, Theorem 9.4.2]{Langer}, we can find the following:
\begin{itemize}
\item If $p \in {\rm Sing}(C)$ is a node, then $e_{orb}\bigg(p;\mathbb{P}^{2}_{\mathbb{C}}, \frac{1}{2} C\bigg) = \frac{1}{4}$.
\item If $q \in {\rm Sing}(C)$ is an ordinary triple point, then $e_{orb}\bigg(p;\mathbb{P}^{2}_{\mathbb{C}}, \frac{1}{2} C\bigg) = \frac{1}{16}$.
\item If $r \in {\rm Sing}(C)$ is an ordinary quadruple point, then $e_{orb}\bigg(p;\mathbb{P}^{2}_{\mathbb{C}}, \frac{1}{2} C\bigg) = 0$ which follows from the fact that the log canonical threshold for ordinary quadruple points is equal to $\frac{1}{2}$, see Remark \ref{lct}.
\end{itemize}
Observe that the left-hand side of $(\star)$ has the following form:
$$3n_{2}\cdot(1 - 1/4) + 3n_{3}\cdot(3/2 + 1 - 1/16) + 3n_{4}\cdot 5 = \frac{9}{4}n_{2} + \frac{117}{16}n_{3} + 15n_{4},$$
hence
$$\frac{9}{4}n_{2} + \frac{117}{16}n_{3} + 15n_{4} \leq \frac{5}{4}m^{2} - \frac{3}{2}m,$$
so after multiplying by $4$ we finally get
$$9n_{2} + \frac{117}{4}n_{3} + 60n_{4} \leq 5m^{2}-6m,$$
which completes the proof.
\end{proof}

\section*{Conflict of Interests}
We declare that there is no conflict of interest regarding the publication of this paper.
\section*{Data Availability Statement}
We do not analyse or generate any datasets, because this work proceeds within a theoretical and mathematical approach. 
\section*{Acknowledgement}
I want to thank Tomasz Pe\l ka for discussions about conic-line arrangements. I would like to thank the anonymous referees for all their comments and suggestions that improved the presentation of this paper.

\vskip 0.5 cm

\bigskip
Piotr Pokora,
Department of Mathematics,
University of the National Education Commission Krakow,
Podchor\c a\.zych 2,
PL-30-084 Krak\'ow, Poland. \\
\nopagebreak
\textit{E-mail address:} \texttt{piotrpkr@gmail.com, piotr.pokora@uken.krakow.pl}
\bigskip
\end{document}